\documentclass{amsart}

\usepackage{amssymb,amscd,amsmath,hyperref,color,enumerate,tikz-cd}
\usepackage[all,cmtip]{xy}

\def\N{\mathbb{N}}
\def\C{\mathbb{C}}
\def\R{\mathbb{R}}
\def\F{\mathbb{F}}
\def\K{\mathbb{K}}
\def\FF{\mathrm{F}}
\def\E{\mathrm{End}}
\def\El{\mathcal{E}\ell}
\def\la{\lambda}
\def\a{\alpha}
\def\b{\beta}
\def\lan{\langle}
\def\ran{\rangle}
\def\sp{\mathrm{span}}
\def\J{\mathcal{J}}
\def\i{\mathrm{id}}
\allowdisplaybreaks[2]

\newtheorem{theorem}{Theorem}[section]
\newtheorem{proposition}[theorem]{Proposition}
\newtheorem{lemma}[theorem]{Lemma}
\newtheorem{corollary}[theorem]{Corollary}
\theoremstyle{definition}
\newtheorem{remark}[theorem]{Remark}
\newtheorem{definition}[theorem]{Definition}
\newtheorem{example}[theorem]{Example}

\numberwithin{equation}{section}

\begin{document}

\title[On the tensor product of nearly simple algebras]{On the ideal structure of the tensor product of nearly simple algebras}

\author{Ilja Gogi\'c}

\date{\today}

\address{Department of Mathematics, Faculty of Science, University of Zagreb, Bijeni\v{c}ka 30,
10000 Zagreb, Croatia}

\email{ilja@math.hr}

\thanks{This work has been fully supported by the Croatian Science Foundation under the project IP-2016-06-1046.}

\thanks{The author is grateful to Matej Bre\v{s}ar for his comments.}

\keywords{nearly simple algebra, tensor product of algebras, ideals}

\subjclass[2010]{Primary 15A69, 16D25   Secondary 16N60}

\begin{abstract}
We define a unital algebra $A$ over a field $\mathbb{F}$ to be nearly simple if $A$ contains a unique non-trivial ideal $I_A$ such that $I_A^2 \neq \{0\}$. If $A$ and $B$ are two nearly simple algebras, we consider the ideal structure of their tensor product $A \otimes B$. The obvious non-trivial ideals of  $A \otimes B$ are:
$$I_A \otimes I_B, \quad I_A \otimes B, \quad A \otimes I_B, \quad \mbox{and} \quad I_A \otimes B + A \otimes I_B.$$ 
The purpose of this paper is to characterize when are all non-trivial ideals of $A \otimes B$ of the above form.
\end{abstract}

\maketitle

\section{Introduction}

Let $A$ and $B$ be unital algebras over a field $\F$. If $A$ is a central simple algebra, it is well-known that all ideals of the tensor product $A \otimes B$ are 
of the form $A \otimes J$, where $J$ is an ideal of $B$ (see e.g. \cite[Theorem~4.42]{INCA} and the comment following its proof).

However, the ideal structure of $A \otimes B$ is generally much more complicated than the one of $A$ and $B$, even in the simplest cases when $A$ and $B$ are proper field extensions of $\F$. In fact, if $\K$ is any proper field extension of  $\F$, then
$\K \otimes_\F \K$ is never a field, since for any $x \in \K\setminus \F$ the non-zero tensor $1 \otimes x - x \otimes 1$ lies in the kernel of the multiplication $m: \K \otimes_\F \K \to \K$, $m(x \otimes y)=xy$. Moreover, the problem of characterizing when is the tensor product of two fields a field (or a domain) is highly non-trivial and for results on this subject we refer to \cite{Cla} and the references within. We also refer to a survey paper \cite{HTY} that considers which properties of commutative algebras $A$ and $B$ are conveyed to $A \otimes B$.

\smallskip

In this paper we study the ideal structure of the tensor product of two unital algebras that both contain only one non-trivial ideal. To avoid  pathologies, we also add one additional requirement:

\begin{definition}
We define a unital algebra $A$ to be \emph{nearly simple} if $A$ contains a unique non-trivial ideal, denoted by $I_A$, such that $I_A^2 \neq \{0\}$.
\end{definition}

The basic examples of nearly simple algebras are the unitizations of non-unital simple algebras (Example \ref{ex:unit}). Further, if $V$ is a vector space over $\F$ of countably infinite dimension, then the algebra $A=\mathrm{End}_{\mathbb{F}}(V)$ of all linear operators on $V$ is nearly simple, where $I_A$ is the ideal of finite rank operators (Example \ref{remcc}). 

If $A$ and $B$ are two nearly simple algebras, then the obvious non-trivial ideals of $A \otimes B$ are:
\begin{equation}\label{admid}
I_A \otimes I_B, \quad I_A \otimes B, \quad A \otimes I_B, \quad \mbox{and} \quad I_A \otimes B + A \otimes I_B.
\end{equation}

The main result of this paper is Theorem \ref{wltp}, in which we characterize when are all non-trivial ideals of $A \otimes B$ of the above form. 

\section{Preliminaries}

Throughout this paper $\F$ denotes a field. Unless specified otherwise, our vector spaces (algebras) will be over $\F$ and all tensor products will be over $\F$.  Also, our algebras are assumed to be associative. 

Given any
algebra $A$, we write $Z(A)$ for the centre of $A$. For $x,y \in A$, the commutator $xy-yx$ is denoted by $[x,y]$. By an ideal of $A$ we always mean a two-sided ideal. As usual, we say that an ideal $I$ of $A$ is non-trivial if $I\neq \{0\}$ and $I \neq A$. If $A$ is unital and $Z(A)=\F 1$, $A$ is said to be \emph{central}.

For an element $a \in A$ by $\lan a \ran$ we denote the principle ideal generated by $a$. Further, for $a,b \in A$ we define a \emph{two-sided multiplication}
$$M_{a,b} : A \to A \qquad \mbox{by} \qquad M_{a,b}:x \mapsto axb.$$ 
By an \emph{elementary operator} on $A$ we mean a map $\phi : A \to A$ that can be expressed as a finite sum of two-sided multiplications, that is $$\phi(x)= \sum_{i}a_ixb_i$$ for some finite collections of  $a_i,b_i\in A$ (the coefficients of $\phi$). We denote the set of all elementary operators on $A$ by $\El(A)$.

For a prime algebra $A$, by $M(A)$ and $Q_s(A)$ we respectively denote the \emph{multiplier algebra} and the \emph{symmetric algebra of quotients} of $A$ (see e.g. \cite{AM,BMM}). The centre of $Q_s(A)$, denoted by $C(A)$, is called the \emph{extended centroid} and it is a field \cite[Corollary~2.1.9]{AM}. A unital prime algebra $A$ is said to be \emph{centrally closed} if $C(A)=Z(A)=\F 1$. In particular, a unital simple algebra is centrally closed if and only if it is central.

If $V$ is a vector space and $L$ a subspace of $V$, then a finite subset $\{v_1, \ldots, v_n\}$ of $V$ is said to be \emph{independent over} $L$ if the set $\{v_1+L, \ldots, v_n+L\}$ is linearly independent in $V/L$. 

We will frequently use the next two well-known facts, but as we have been unable to find a direct reference we include their proofs for completeness. 

\begin{lemma}\label{tendec}
Let $V$ and $W$ be vector spaces and let $L$ be a subspace of $V$. Assume that $$t=\sum_{i=1}^n v_i \otimes w_i \in V \otimes W$$ is a tensor of rank $n\geq 1$ such that $v_i \notin L$ for some $1 \leq i\leq n$. Then there are $1 \leq k\leq n$, $v_1',\ldots, v_n' \in V$ and $w_1',\ldots, w_n' \in W$ such that
$$
t=\sum_{i=1}^n v_i' \otimes w_i',
$$
where $\{v_1', \ldots, v'_k\}$ is independent over $L$ and $v'_{k+1}, \ldots, v'_{n}\in L$.
\end{lemma}
\begin{proof}
Without loss of generality we may assume that $\{v_1, \ldots, v_k\}$ is a maximal subset of $\{v_1, \ldots, v_n\}$ that is independent over $L$. If $k=n$ we are done, so assume that $k <n$.
Then for each $j=k+1, \ldots, n$ there is $f_j \in L$ and scalars $\la_{1j}, \ldots , \la_{kj}\in \F$ that are not all zero such that
$$v_j=\sum_{i=1}^k \lambda_{ij} v_i + f_j.$$
Then
$$t=\sum_{i=1}^k v_i \otimes \left(w_i +\sum_{j=k+1}^n \lambda_{ij} w_j \right)+ \sum_{j=k+1}^n f_j \otimes w_j$$
is a desired decomposition of $t$.
\end{proof}

\begin{proposition}\label{prop:quot}
Let $A$ and $B$ be algebras. If $I$ and $J$ are ideals of $A$ and $B$, respectively,  with corresponding canonical maps $q_I : A \to A/I$ and $q_J: B\to B/J$, then the map 
$$q_I \otimes q_J : A \otimes B \to (A/I)\otimes (B/J), \qquad \sum_i a_i \otimes b_i \mapsto \sum_i (a_i+I)\otimes (b_i+J)$$
is an algebra epimorphism with $\ker(q_I \otimes q_J)=I \otimes B + A \otimes J$. In particular,
$$(A \otimes B)/(I \otimes B + A \otimes J)\cong (A/I)\otimes (B/J),$$
as algebras.
\end{proposition}
\begin{proof}
Obviously, $q_I \otimes q_J$ is an algebra epimorphism and $I \otimes B + A \otimes J \subseteq \ker(q_I \otimes q_J)$. For the reverse inclusion, assume that 
$$t=\sum_{i=1}^n a_i \otimes b_i \in A \otimes B$$ is a tensor of rank $n\geq 1$ such that $(q_I \otimes q_J)(t)=0$. If all $a_i$ belong to $I$ and all $b_i$ belong to $J$ we have nothing to prove. Assume that some $a_i \notin I$.
By Lemma \ref{tendec} we may assume that $\{a_1, \ldots, a_k\}$ is independent over $I$ and $a_{k+1}, \ldots, a_n \in I$ for some $1\leq k\leq n$. Then
$$\sum_{i=1}^k (a_i+I) \otimes (b_i +J)=(q_I \otimes q_J)(t)=0,$$
which forces $b_1, \ldots, b_k \in J$.  Hence,
$$t= \sum_{j=k+1}^n a_j \otimes b_j + \sum_{i=1}^k a_i \otimes b_i  \in I \otimes B + A \otimes J.$$
The similar argument also shows that $t \in I \otimes B + A \otimes J$ if some $b_i \notin J$.
\end{proof}

\section{Results}

We begin this section with the next simple observation.

\begin{proposition}\label{prop:bas}
Let $A$ be a nearly simple algebra. Then:
\begin{itemize}
\item[(a)] $A$ is prime.
\item[(b)] $I_A$ is a simple infinite-dimensional algebra such that $Z(I_A)=\{0\}$. 
\item[(c)] $Z(A)$ is a field and $Q_s(A)=M(I_A)$.
\end{itemize}
\end{proposition}
\begin{proof}
(a) Since the only non-trivial ideal $I_A$ of $A$ satisfies $I_A^2\neq \{0\}$, $A$ is obviously prime.

(b) Assume that $J$ is a non-zero ideal of $I_A$. Since $A$ is prime, $I_A$ is also prime (see e.g. \cite[Lemma~1.1.3]{AM}). In particular, $J^3 \neq \{0\}$, so $I_A J I_A$ is a non-zero ideal of $A$ that is contained in $J$. Then, since $A$ is nearly simple, we have $I_A=I_A J I_A\subseteq J$, and thus $J=I_A$. This shows that $I_A$ is simple.

Assume that $I_A$ is unital with unity $e$. Then by \cite[Lemma 2.54]{INCA} $e$ is a central idempotent of $A$ such that $I_A=eA$. Then $(1-e)A$ is also a non-trivial ideal of $A$, so $I_A=(1-e)A$ and thus $I_A=\{0\}$; a contradiction. Hence, $I_A$ is a non-unital simple algebra and consequently $Z(I_A)=\{0\}$.

Now assume that $I_A$ is finite-dimensional. Then, since $I_A$ is simple, Wedderburn's Theorem (see e.g. \cite[Theorem~2.61]{INCA}) implies that there is a natural number $n$ and a division algebra $\mathbb{D}$ over $\F$ such that $I_A \cong \mathrm{M}_n(\mathbb{D})$. In particular, $I_A$ is unital; a contradiction with the preceding paragraph.

(c) Assume that $z \in Z(A)$ is a non-invertible element. Then $zA$ is an ideal of $A$ such that $zA \neq A$ and hence $zA\subseteq I_A$. In particular, by (b), $$z \in Z(A)\cap I_A\subseteq Z(I_A)=\{0\},$$ 
that is, $z=0$. Thus, $Z(A)$ is a field. Finally, the equality $Q_s(A)=M(I_A)$ is a direct consequence of  \cite[Proposition~2.1.3]{AM} and the fact that $I_A^2=I_A$. 
\end{proof}

We now present some examples of nearly simple algebras.

\begin{example}\label{ex:unit}
Let $A$ be a non-unital simple algebra.  Then its unitization $A^\sharp=\F \times A$ (see e.g. \cite[Section~2.3]{INCA}) is a nearly simple algebra with $I_{A^\sharp}=A$. 

Similarly, set
$$B:=C(A)+A\subseteq Q_s(A).$$
As $A$ is simple and non-unital, we have $Q_s(A)=M(A)$ and  $C(A)\cap A=Z(A)=\{0\}$.  Hence, since $A$ is an ideal of $M(A)$, $B$ is a subalgebra of $M(A)$ such that $B/A\cong C(A)$ (which is a field). Thus, $B$ is a nearly simple algebra with $I_B=A$.
\end{example}

\begin{example}\label{remcc}
Let $V$ be a vector space over $\F$ of countably infinite dimension. Consider the algebra $\E_\F(V)$ of all linear operators on $V$. If by $\FF(V)$ we denote the ideal of finite rank operators in $\E_\F(V)$, it is well-known that $\FF(V)$ is the only non-trivial ideal of $\E_\F(V)$ 
%(see. e.g. \cite[Corollary 3.4]{GP}) 
and that $\E_\F(V)=M(\FF(V))$. In particular, by Proposition \ref{prop:bas} (c), $Q_s(\E_\F(V))=\E_\F(V)$, so $\E_\F(V)$ is a nearly simple centrally closed algebra (see also \cite[Example 7.28]{INCA}).

Further, if $D$ is any simple subalgebra of $\E_\F(V)$ that contains the identity operator $1$ (e.g. one of Weyl algebras), define 
$$A:=D+\FF(V)\subseteq \E_\F(V).$$
As $D$ is simple, we have $D \cap \FF(V)=\{0\}$, so $A/\FF(V)\cong D$ is simple and thus $\FF(V)$ is the unique non-trivial ideal of $A$. Hence, $A$ is a nearly simple algebra. Further, since by Proposition \ref{prop:bas} (c) $Q_s(A)=M(\FF(V))=\E_\F(V)$, $C(A)=Z(\E_\F(V))=\F 1$, so $A$ is also centrally closed.
\end{example}

\begin{definition}
If $A$ and $B$ are nearly simple algebras, we say that an ideal $\mathcal{J}$ of  $A \otimes B$ is \emph{admissible} if $\mathcal{J}$ is either trivial or of the form as in (\ref{admid}).
\end{definition}

\begin{lemma}\label{lem:std}
Assume that $A$ and $B$ are two nearly simple algebras such that all ideals of $A \otimes B$ are admissible. Then the algebra $(A/I_A) \otimes B$ is nearly simple, whose only non-trivial ideal is $(A/I_A) \otimes I_B$. Similarly, $A \otimes (B/I_B)$ is nearly simple with a non-trivial ideal  $I_A \otimes (B/I_B)$.
\end{lemma}
\begin{proof}
Let $q_{I_A} : A \to A/I_A$ be the canonical map and consider the algebra epimorphism
$q_{I_A} \otimes \i_B : A \otimes B \to (A/I_A) \otimes B$. Let $\J$ be a non-trivial ideal of $(A/I_A) \otimes B$. Then, since all ideals of $A \otimes B$ are admissible, and $\ker(q_{I_A} \otimes \i_B)=I_A \otimes B$ (Proposition \ref{prop:quot}), $(q_{I_A} \otimes \i_B)^{-1}(\J)$ is an ideal of $A \otimes B$ that strictly contains $I_A \otimes B$ and so $I_A \otimes B +A \otimes I_B \subseteq (q_{I_A} \otimes \i_B)^{-1}(\J)$. Also, since  $\J$ is non-trivial, $\J \neq (A/I_A)\otimes B$, so $(q_{I_A} \otimes \i_B)^{-1}(\J)\neq A \otimes B$.
Consequently,
$$(q_{I_A} \otimes \i_B)^{-1}(\J)=I_A \otimes B +A \otimes I_B$$
and thus
\begin{eqnarray*}
(A/I_A) \otimes I_B &=& (q_{I_A} \otimes \i_B)(I_A \otimes B + A \otimes I_B)\\
&=&  (q_{I_A} \otimes \i_B)((q_{I_A} \otimes \i_B)^{-1}(\J)) \\
&=& \J.
\end{eqnarray*}
The similar argument also shows that $I_A \otimes (B/I_B)$ is the only non-trivial ideal of
$A \otimes (B/I_B)$.
\end{proof}

We now record some non-examples which helped us to conjecture the main result of this paper, Theorem \ref{wltp}.

\begin{example}\label{ex:ex1}
Let $V$ be a real vector space of countably infinite dimension.
Consider $\C$ as a unital subalgebra of $\E_\R(V)$. For example, if $\{e_n : \, n \in \N\}$ is a basis for $V$, define a linear operator $T \in \E_\R(V)$ by $T(e_{2n-1})=e_{2n}$ and  $T(e_{2n})=-e_{2n-1}$ for all $n \in \N$. Then obviously $T^2=-1$, where $1$ is the identity operator, so we can identify $\C$ with the subalgebra $\{\alpha 1+ \beta T : \, \alpha, \beta \in \R\}$ of $\E_\R(V)$.
Set 
$$A:=\C+\FF(V)\subseteq \E_\R(V).$$
By Example \ref{remcc} $A$ is a centrally closed nearly simple (real) algebra with $I_A=\FF(V)$. Consider the tensor product $A \otimes A$. Since by Proposition \ref{prop:quot}
$$(A \otimes A)/(\FF(V) \otimes A + A \otimes \FF(V))\cong (A/\FF(V))\otimes (A/\FF(V))\cong \C \otimes_\R \C$$
and since $\C \otimes_\R  \C$ is not a field, we conclude that $\FF(V) \otimes A + A \otimes \FF(V)$ is not a maximal ideal of $A \otimes A$. In fact, since $\C \otimes_\R \C \cong \C \oplus \C$  (see e.g. \cite[Example 4.45]{INCA}), $\FF(V) \otimes A + A \otimes \FF(V)$ is contained in two distinct maximal ideals of $A \otimes A$.
\end{example}

\begin{example}\label{ex:ex2}
Let $W$ be a complex vector space of countably infinite dimension. Consider the real algebra
$$B:=\R1 + \FF(W) \subseteq \E_\C(W),$$
where $1$ is the identity operator. Then $B$ is a central nearly simple algebra whose only non-trivial ideal is $\FF(W)$.  Note that $B$ is not centrally closed, since by Proposition \ref{prop:bas} (c) $Q_s(B)=M(\FF(W))=\E_\C(W)$ and thus 
$C(B)=Z(\E_\C(W))=\C 1$ (see also \cite[Example 7.37]{INCA}). Let $c_1$ and $c_2$ be elements of $\C \otimes_\R \C$ defined by $c_1:=1\otimes i + i \otimes 1$ and $c_2:=1\otimes i - i \otimes 1$. Then $$\mathcal{J}_1:=c_1 (\FF(W)\otimes \FF(W))\qquad \mbox{and} \qquad \mathcal{J}_2:=c_2 (\FF(W)\otimes \FF(W))$$ are two non-zero ideals of $B \otimes  B$ such that $\J_1 \J_2=\{0\}$. In particular, $\mathcal{J}_1$ and $\mathcal{J}_2$ are non-admissible and $B \otimes B$ is not even prime.
\end{example}

\begin{example}\label{ex:ex3}
Consider the tensor product $A \otimes B$ of real algebras $A$ and $B$ from Examples \ref{ex:ex1} and \ref{ex:ex2}. If $c_1,c_2 \in \C\otimes_\R \C$ are as in Example \ref{ex:ex2}, then using the isomorphism $A/\FF(V)\cong \C$, we see that
$$\mathcal{K}_1:=c_1((A/\FF(V))\otimes \FF(W)) \qquad \mbox{and} \qquad \mathcal{K}_2:=c_2((A/\FF(V))\otimes \FF(W))$$ are two non-zero ideals of $(A/\FF(V))\otimes B$  such that
$\mathcal{K}_1\mathcal{K}_2=\{0\}$. In particular, $(A/\FF(V))\otimes B$  is not prime, so by Lemma \ref{lem:std} $A \otimes B$ has a non-admissible ideal.
\end{example}

We now state the main result of the paper.

\begin{theorem}\label{wltp}
Let $A$ and $B$ be two nearly simple algebras. Then all ideals of  $A \otimes B$ are admissible if and only if all tensor products
\begin{equation}\label{eqtp}
Z(A/I_A)\otimes Z(B/I_B), \ \ C(A)\otimes Z(B/I_B), \ \ Z(A/I_A)\otimes C(B),  \ \ C(A)\otimes C(B)
\end{equation}
are fields.
\end{theorem}

\begin{remark}
In Examples \ref{ex:ex1}, \ref{ex:ex2} and \ref{ex:ex3} (respectively), when considering $A \otimes A$, $B \otimes B$ and $A \otimes B$ (respectively), all tensor products in (\ref{eqtp}) are fields, except $Z(A/I_A)\otimes Z(A/I_A)$, $C(B)\otimes C(B)$ and  $Z(A/I_A)\otimes C(B)$ (respectively). This in particular demonstrates that none of the assumptions of Theorem \ref{wltp} cannot be omitted (if $A$ and $B$ are algebras from Examples \ref{ex:ex1} and \ref{ex:ex2}, then by symmetry $B \otimes A$ has a non-admissible ideal, $C(B)\otimes Z(A/I_A)$ is not a field, while $Z(B/I_B) \otimes Z(A/I_A)$, $Z(B/I_B)\otimes C(A)$ and $C(B)\otimes C(A)$ are fields).
\end{remark}

The proof of Theorem \ref{wltp} heavily relies on the main result of \cite{NW} (see also \cite{JP}) and its consequence which we state below.

\begin{theorem}\cite[Theorem]{NW}\label{elten}
If $A$ and $B$ are prime algebras, then each non-zero ideal of $A \otimes B$ contains a non-zero elementary tensor if and only if $C(A)\otimes C(B)$ is a field.
\end{theorem}

\begin{corollary}\label{tpsimple}
If $A$ and $B$ are unital simple algebras, then $A \otimes B$ is a simple algebra if and only if $Z(A)\otimes Z(B)$ is a field.
\end{corollary}

\begin{proof}
If  $A \otimes B$ is simple, then $Z(A) \otimes Z(B) \cong Z(A \otimes B)$ (see e.g. \cite[Corollary~4.32]{INCA}) is a field. 

Conversely, since both $A$ and $B$ are unital and simple, we have $Q_s(A)=A$ and $Q_s(B)=B$, so in particular $C(A)=Z(A)$ and $C(B)=Z(B)$. Hence if $Z(A)\otimes Z(B)$ is a field, then $A \otimes B$ is simple by \cite[Corollary~1]{NW}.
\end{proof}

\begin{lemma}\label{lelop}
Let $A$ and $B$ be algebras and let $\mathcal{J}$ be an ideal of $A\otimes B$.
If $\phi\in \El(A)$ and $\psi \in \El(B)$ are elementary operators, then $(\phi \otimes \psi)(\mathcal{J})\subseteq \mathcal{J}$. 

In particular, if both $A$ and $B$ are unital and $a\otimes b \in \J$ for some $a \in A$ and $b \in B$, then $\lan a \ran \otimes \lan b \ran \subseteq \J$.
\end{lemma}
\begin{proof}
Assume that $\phi=\sum_i M_{a_i,a_i'}$ and $\psi=\sum_j M_{b_j,b_j'}$ for some finite collections of $a_i,a_i'\in A$ and $b_j,b_j' \in B$. Then
$$\phi \otimes \psi =\sum_j \sum_i M_{a_i \otimes b_j, a_i' \otimes b_j'}\in \El(A \otimes B)$$
and thus $(\phi \otimes \psi)(\mathcal{J})\subseteq \mathcal{J}$.

Next, assume that both $A$ and $B$ are unital. If $a \otimes b \in \J$ for some $a \in A$ and $b \in  B$, then for any $x \in \lan a \ran$ and $y \in \lan b \ran$ there are elementary operators $\phi \in \El(A)$ and $\psi \in \El(B)$ such that $\phi(a)=x$ and $\psi(b)=y$. Hence,   
$$x \otimes y =(\phi \otimes \psi)(a \otimes b) \in \J$$
and consequently $\lan a \ran \otimes \lan b \ran \subseteq \J$.
\end{proof}

\begin{proposition}\label{propmax}
Let $A$ and $B$ be unital prime algebras.
\begin{itemize}
\item[(a)] If both $A$ and $B$ contain the smallest non-zero ideals $I$ and $J$, respectively, then $I \otimes J$ is the smallest non-zero ideal of $A \otimes B$ if and only if  $C(A)\otimes C(B)$ is a field.
\item[(b)] If $M$ and $N$ are maximal ideals of $A$ and $B$, respectively, then $M \otimes B+ A \otimes N$ is a maximal ideal of  $A \otimes B$ if and only if  $Z(A/M)\otimes Z(B/N)$ is a field.
\end{itemize}
\end{proposition}

\begin{proof}
(a) Assume that $C(A)\otimes C(B)$ is a field and let $\mathcal{J}$ be a non-zero ideal of $A \otimes B$. By Theorem \ref{elten}, $\mathcal{J}$ contains a non-zero elementary tensor $a \otimes b$. By Lemma \ref{lelop} we have $\lan a \ran \otimes \lan b \ran \subseteq \J$. By assumption, $I \subseteq \lan a \ran$ and $J \subseteq \lan b \ran$, so $I \otimes J \subseteq \lan a \ran \otimes \lan b \ran  \subseteq \mathcal{J}$.

If, on the other hand, $C(A)\otimes C(B)$ is not a field, choose a non-zero non-invertible element $c \in C(A)\otimes C(B)$. Since $I$ is the smallest non-zero ideal of $A$, we have $C(A)I \subseteq A$. Similarly, $C(B)J\subseteq B$. Then, by the proof of \cite[Theorem]{NW}, $c(I \otimes J)$ defines a non-zero ideal of $A \otimes B$ that does not contain a non-zero elementary tensor. In particular, $c(I \otimes J)$ cannot contain $I \otimes J$.

(b) Obviously $M \otimes B + A \otimes N$ is a maximal ideal of $A \otimes B$ if and only if $(A\otimes B)/(M \otimes B + A \otimes N)$ is a simple algebra. Since by Proposition \ref{prop:quot}
$$
(A \otimes B)/(M \otimes B+A \otimes N) \cong (A/M) \otimes (B/N),
$$
by Corollary \ref{tpsimple} $(A\otimes B)/(M \otimes B + A \otimes N)$ is simple if and only if  $Z(A/M) \otimes Z(B/N)$ is a field.
\end{proof}

In the proof of Theorem \ref{wltp} we will use the next version of Amitsur's Lemma (see \cite[Theorem 4.2.7]{BMM}) which states
that if $T_1,\dots, T_n$ are linear operators between vector spaces $V$ and $W$ such that the vectors  $T_1(x),\dots,T_n(x)$ are linearly dependent for every $x\in V$, then a non-trivial linear combination of $T_1, \dots,T_n$ has a finite rank.

We will also use the next fact, which was proved in \cite{BG} (see also \cite{Be,BE}).

\begin{lemma} \cite[Lemma~3.5]{BG}\label{l2}
Let $\delta$ be a non-zero derivation of a simple algebra $D$. If $\delta$ has a finite rank, then $D$ is finite-dimensional.
\end{lemma}

%\begin{proof}
%Take $a\in D$ such that $\d(a)\ne 0$. Since $D$ is simple, $\d(a)b\d(a)\ne 0$ for some  $b\in D$.
%For any $x,y\in D$,
%$$x\d(a)b\d(a)y= \big(\d(xa)-\d(x)a\big)b \big(\d(ay)-a\d(y)\big).$$
%Thus,
%$$x\d(a)b\d(a)y\in (\d(D)+\d(D)a)b(\d(D) + a\d(D)).$$
%Since $\d(D)$ is a finite-dimensional space, it follows that the ideal generated by $\d(a)b\d(a)$ is finite-dimensional. As $D$ is simple, this gives the desired conclusion.
%\end{proof}

\begin{proof}[Proof of Theorem \ref{wltp}]
First assume that all ideals of $A \otimes B$ are admissible. Proposition \ref{propmax} then implies that $C(A)\otimes C(B)$ and $Z(A/I_A)\otimes Z(B/I_B)$ are fields. Further, by Lemma \ref{lem:std}, the only non-trivial ideal of $(A/I_A) \otimes B$ is $(A/I_A)\otimes I_B$. In particular, since $(A/I_A) \otimes I_B$ contains non-zero elementary tensors, by  Theorem \ref{elten} $Z(A/I_A)\otimes C(B)$ is a field. A similar argument also shows that $C(A)\otimes Z(B/I_B)$ must be a field.

\smallskip

Now assume that all tensor products in (\ref{eqtp}) are fields and let $\mathcal{J}$ be a non-zero ideal of $A \otimes B$. By Proposition \ref{propmax} (a) we have $I_A \otimes I_B \subseteq \mathcal{J}$.
Assume that $$I_A \otimes I_B \subsetneq  \mathcal{J}.$$
If $q_{I_A} : A \to A/I_A$ and $q_{I_B} : B \to B/I_B$ are the canonical maps, then one of the ideals
$$ (q_{I_A} \otimes \i_B)(\J) \subseteq (A/I_A) \otimes B \qquad \mbox{or} \qquad (\i_A \otimes q_{I_B})(\J) \subseteq A \otimes (B/I_B)$$
must be non-zero, since otherwise
\begin{eqnarray*}
\J &\subseteq& \ker (q_{I_A} \otimes \i_B) \cap \ker (\i_A \otimes q_{I_B}) = (I_A \otimes B) \cap (A \otimes I_B) \\
&=& I_A \otimes I_B.
\end{eqnarray*}
Assume that $(q_{I_A} \otimes \i_B)(\J)$ is a non-zero ideal of $(A/I_A)\otimes B$. By assumption, $Z(A/I_A) \otimes C(B)$ is a field, so by Theorem \ref{elten} $(q_{I_A} \otimes \i_B)(\J)$ contains a non-zero elementary tensor. 

Let $n \geq 1$ be the smallest number with the property that there exists a tensor $t \in \J$ of rank $n$ such that $(q_{I_A} \otimes \i_B)(t)$ is a non-zero elementary tensor in $(A/I_A)\otimes B$. We claim that $n=1$, so that  $a \otimes b \in \J$ for some $a \in A \setminus I_A$ and $b \in B \setminus \{0\}$. In this case, $\lan a \ran =A$ and $I_B \subseteq \lan b \ran$, so by Lemma \ref{lelop} $ \lan a \ran \otimes \lan b \ran\subseteq \J$. In particular,
\begin{equation}\label{eq:inq}
A \otimes I_B \subseteq \J.
\end{equation}

In order to obtain a contradiction, assume that $n>1$. Let $t \in \J$ be any tensor of rank $n$ for which there exist $a' \in A \setminus I_A$ and $b' \in B\setminus\{0\}$ such that
\begin{equation}\label{eq:t1}
(q_{I_A} \otimes \i_B)(t)=(a'+I_A)\otimes b'.
\end{equation}
If $t$ is represented as
\begin{equation}\label{eq:t2}
t=\sum_{i=1}^n a_i \otimes b_i,
\end{equation}
then obviously not all $a_i$ belong to $I_A$. By Lemma \ref{tendec} we may assume that the set $\{a_1, \ldots, a_k\}$ is independent over $I_A$ and that $a_{k+1},\ldots, a_n \in I_A$ for some $1 \leq k \leq n$. Also, since $t$ is of rank $n$, the set $\{b_1, \ldots, b_n\}$ is linearly independent. We first show that $k=1$. Indeed, by (\ref{eq:t1}) and (\ref{eq:t2}) we have
\begin{equation}\label{eq:t3}
\sum_{i=1}^k (a_i + I_A) \otimes b_i = (a'+I_A)\otimes b'
\end{equation}
in $(A/I_A) \otimes B$. Clearly, $b' \in \sp\{b_1, \ldots, b_k\}$, since otherwise the set
$\{b', b_1, \ldots, b_k\}$ would be linearly independent and consequently $a_1, \ldots, a_k \in I_A$; a contradiction. Hence, there are scalars $\la_1, \ldots, \la_k \in \F$ such that $$b'=\sum_{i=1}^k \la_i b_i.$$
Then, by (\ref{eq:t3}),
$$\sum_{i=1}^k (a_i- \la_i a' + I_A) \otimes b_i=0,$$
which forces $a_i- \la_i a' \in I_A$ for all $i=1, \ldots,k$. Since the set $\{a_1, \ldots, a_k\}$ is independent over $I_A$, this is only possible if $k=1$, as claimed. Hence $a_1 \notin I_A$, while $a_2, \ldots, a_n \in I_A$.

Further, without loss of generality we may assume that $a_1=1_A$. Indeed, since $a_1 \notin I_A$, we have $\lan a_1 \ran=A$,  so there is an  elementary operator $\phi\in \El(A)$ such that $\phi(a_1)=1_A$. For $2 \leq i \leq n$ set $a_i':=\phi(a_i)$. Then, $a_2', \ldots , a_n' \in I_A$ and by Lemma \ref{lelop} we have
$$t':=1_A \otimes b_1 + \sum_{i=2}^n a_i' \otimes b_i = (\phi \otimes \mathrm{id}_B)(t)\in \mathcal{J}.$$
Obviously,  $(q_{I_A}\otimes \i_B)(t')=(1_A+ I_A) \otimes b_1 \neq 0$ in $(A/I_A) \otimes B$, so by minimality of $n$,  $t'$ is also a tensor of rank $n$. So if $a_1 \neq 1_A$, we may substitute $t$ by $t'$. Hence, in the sequel we assume that
\begin{equation}\label{eq:31}
t=1_A \otimes b_1 + \sum_{i=2}^n a_i \otimes b_i \in \J,
\end{equation}
where $a_2, \ldots, a_n \in I_A$.

Set $\K:=C(A)$. By Proposition \ref{prop:bas} (c), we have $C(A)=Z(M(I_A))$, so $\K I_A \subseteq I_A$. Hence, we may consider $I_A$ as an algebra over $\K$. Without loss of generality we may assume that $\{a_2, \ldots, a_l\}$ is a maximal $\K$-independent subset of $\{a_2, \ldots, a_n\}$. Then for each $j=l+1, \ldots, n$ there are $\a_{ij} \in \K$ such that
$$a_j=\sum_{i=2}^l \a_{ij} a_i,$$
so by (\ref{eq:31})
\begin{equation}\label{eq:t4}
t=1_A \otimes b_1 + \sum_{i=2}^l a_i \otimes b_i + \sum_{j=l+1}^n \left(\sum_{i=2}^l \a_{ij} a_i\right) \otimes b_j.
\end{equation}

We claim there is an element $x_0 \in I_A$ such that the set $\{[a_2,x_0], \ldots, [a_l,x_0]\}$ is $\K$-independent. Indeed, if this set would be $\K$-dependent for all $x_0 \in I_A$, then by Amitsur's Lemma there are $\b_2, \ldots , \b_l \in \K$ which are not all zero such that the inner derivation $\delta : I_A \to I_A$ defined by
$$\delta(x):=[\b_2 a_2 + \ldots + \b_l a_l, x]$$
has a finite rank. Since by Proposition \ref{prop:bas} (b) $I_A$ is simple and infinite-dimensional, Lemma \ref{l2} implies that $\delta$ is zero. As $Z(I_A)=\{0\}$ (again by Proposition \ref{prop:bas} (b)), we conclude that
$$\b_2 a_2 + \ldots + \b_l a_l =0.$$
Since the set $\{a_2, \ldots, a_l\}$ is $\K$-independent, this yields $\b_2= \ldots = \b_l=0$; a contradiction.

If $x_0\in I_A$ is an element from the preceding paragraph, by (\ref{eq:t4}) we have
\begin{equation}\label{eq:t5}
\mathcal{J} \ni [t,  x_0 \otimes 1_B]= \sum_{i=2}^l [a_i,x_0] \otimes b_i + \sum_{j=l+1}^n \left(\sum_{i=2}^l \a_{ij} [a_i,x_0]\right) \otimes b_j.
\end{equation}
Since $\{[a_2,x_0], \ldots, [a_l,x_0]\}$ is $\K$-independent, by \cite[Theorem 2.3.3]{BMM} for each $i=2, \ldots, l$ there is an elementary operator $\phi_i\in \El(I_A)$ such that $$\phi_i([a_i,x_0])\neq 0 \qquad  \mbox{and} \qquad \phi_i([a_j,x_0])=0 \qquad \forall j\in \{2, \ldots, l\} \setminus \{i\}.$$
Further, since $0 \neq \phi_i([a_i,x_0])\in I_A$ and since $I_A$ is simple (Proposition \ref{prop:bas} (b)), $I_A \phi_i([a_i,x_0])I_A$ is a non-zero ideal of $I_A$ and thus $I_A \phi_i([a_i,x_0])I_A=I_A$. Hence, for each $i=2, \ldots, l$ there is another elementary operator $\psi_i\in \El(I_A)$ such that 
$$\psi_i(\phi_i([a_i,x_0]))=a_i.$$ 
Set $\theta_i:=\psi_i \circ \phi_i$. Then $\theta_i$ is an elementary operator on $I_A$  such that
$$\theta_i([a_i,x_0])=a_i \qquad \mbox{and} \qquad \theta_i([a_j,x_0])=0 \qquad \forall j\in \{2, \ldots, l\} \setminus \{i\}.$$
By extending $\theta_i$ to an elementary operator on $A$ (with the same coefficients), (\ref{eq:t5}) and Lemma \ref{lelop} imply
$$\mathcal{J}\ni(\theta_i \otimes \mathrm{id}_B)([t,x_0 \otimes 1_B])=a_i \otimes b_i + \sum_{j=l+1}^n (\a_{ij}a_i)\otimes b_j,$$
for each $i=2, \ldots, l$, so by (\ref{eq:t4})
$$\mathcal{J} \ni \sum_{i=2}^l \left(a_i \otimes b_i + \sum_{j=l+1}^n ( \a_{ij}a_i)\otimes b_j \right) =t-1_A \otimes b_1.
$$
Since $t \in \mathcal{J}$, we conclude that $1_A \otimes b_1 \in \mathcal{J}$, which contradicts the assumption that $n>1$. In particular, (\ref{eq:inq}) holds. 

The similar arguments would also show that if $(\i_A \otimes q_{I_B})(\J)$ is a non-zero ideal of $A \otimes (B/I_B)$, then $I_A \otimes B \subseteq \J$.

\smallskip

Finally, if  $\mathcal{J}$ strictly contains $I_A \otimes B$ or $A \otimes I_B$ (respectively), then the above proof shows that $\mathcal{J}$ also contains $A \otimes I_B$ or $I_A \otimes B$ (respectively). In any of those two cases we clearly have $I_A \otimes B + A \otimes I_B \subseteq\mathcal{J}$. Since, by assumption, $Z(A/I_A)\otimes Z(B/I_B)$ is a field, Proposition \ref{propmax} (b) implies that $I_A \otimes B + A \otimes I_B$ is a maximal ideal of $A \otimes B$, which finishes the proof.
\end{proof}

\begin{corollary}\label{cor:univ}
Let $A$ be a nearly simple algebra. The following conditions are equivalent:
\begin{itemize}
\item[(i)] For any nearly simple algebra $B$, all ideals of $A \otimes B$ are admissible.
\item[(ii)] All ideals of $A \otimes A$ are admissible.
\item[(iii)] $A$ is centrally closed and $A/I_A$ is central. 
\end{itemize}
\end{corollary}
\begin{proof}
(i) $\Longrightarrow$ (ii). This is trivial.

(ii) $\Longrightarrow$ (iii). If $A$ would not be centrally closed, $C(A)$ would be a proper field extension of $\F$. But then $C(A)\otimes C(A)$ is not a field, so by Proposition \ref{propmax} (a) $A \otimes A$ has a non-admissible ideal. Similarly, if $A/I_A$ is not central, then $Z(A/I_A)\otimes Z(A/I_A)$ is not a field, so by Proposition \ref{propmax} (b) $A \otimes A$ has a non-admissible ideal.

(iii)  $\Longrightarrow$ (i). By assumption $C(A)\cong Z(A/I_A)\cong \F $, so the assertion follows directly from Theorem \ref{wltp}. 
\end{proof}

\begin{example}
In view of Example \ref{remcc}, if $V$ is a vector space over $\F$ of countably infinite dimension, Corollary \ref{cor:univ} in particular applies to algebras
$\E_\F(V)$ and $D+\FF(V)$, where $D$ is any central simple subalgebra of $\E_\F(V)$ (that $\E_\F(V)/\FF(V)$ is central follows from \cite[Proposition~2.9]{BG}).
\end{example}

\end{document}